\documentclass[a4paper,reqno,12pt]{amsart}
\usepackage{t1enc}
\usepackage[margin=1.0in]{geometry}
\usepackage{ifthen}
\usepackage{graphicx}
\usepackage{xcolor}

\newcommand{\R}{\mathbf{R}}

\newcommand{\pr}{\mathbf{P}}
\newcommand{\ex}{\mathbf{E}}

\theoremstyle{plain}
\newtheorem{theorem}{Theorem}
\newtheorem{lemma}{Lemma}
\newtheorem{corollary}{Corollary}
\newtheorem{proposition}{Proposition}

\theoremstyle{definition}

\theoremstyle{remark}

\newcommand{\formula}[2][nolabel]
{\ifthenelse{\equal{#1}{nolabel}}
 {\begin{align*} #2 \end{align*}}
 {\ifthenelse{\equal{#1}{}}
  {\begin{align} #2 \end{align}}
  {\begin{align} \label{#1} #2 \end{align}}
 }
}

\numberwithin{equation}{section}

\begin{document}

%
%                            ---------- o ----------
%

\title[Sharp estimates of transition probability density for Bessel
process in half-line]{Sharp estimates of transition probability density for Bessel
process in half-line}
\author{Kamil Bogus, Jacek Ma{\l}ecki}
\address{Kamil Bogus, Jacek Ma{\l}ecki \\ Institute of Mathematics and
Computer Science \\ Wroc{\l}aw University of Technology \\ ul.
Wybrze{\.z}e Wyspia{\'n}\-skiego 27 \\ 50-370 Wroc{\l}aw,
Poland
}
\email{kamil.bogus@pwr.wroc.pl, jacek.malecki@pwr.wroc.pl}


\begin{abstract} In this paper we study the Bessel process $R_t^{(\mu)}$ with index $\mu\neq 0$ starting from $x>0$ and killed when it reaches a positive level $a$, where $x>a>0$. We provide sharp estimates of the transition probability density $p_a^{(\mu)}(t,x,y)$ for the whole range of space parameters $x,y>a$ and every $t>0$.
\end{abstract}

\maketitle

\section{Introduction}
\label{section:introduction}
Let $R_t^{(\mu)}$ be the Bessel process with index $\mu\neq 0$. The transition probability density (with respect to the Lebesgue measure) of the process is expressed by the modified Bessel function in the following way
   \formula[eq:transitiondensity:formula]
      {
p^{(\mu)}(t,x,y) = \frac1t \left(\frac{y}{x}\right)^{\mu}y
\exp\left(-\frac{x^2+y^2}{2t}\right)I_{|\mu|}\left(\frac{xy}{t}\right)\/,\quad
x,y,t> 0\/.}
Our main goal is to describe behaviour of densities of the transition probabilities for the process $R_t^{(\mu)}$ killed when it leaves a half-line $(a,\infty)$, where $a>0$. Note that if the process starts from $x>a$ then the first hitting time $T_a^{(\mu)}$ of a level $a$ is finite a.s. when $\mu<0$ but it is infinite with positive probability when $\mu>0$. The density kernel of the killed semi-group is given by the Hunt formula
     \formula[eq:hunt:general]{
p_a^{(\mu)}(t,x,y) &= p^{(\mu)}(t,x,y)-\textbf{E}_x^{(\mu)}[t>T^{(\mu)}_a; p^{(\mu)}(t-T^{(\mu)}_a, R^{(\mu)}_{T^{(\mu)}_a},y)]\/,
}
where $x,y>a$ and $t>0$. The main result of the paper is given in 
\begin{theorem}
   \label{thm:main}
   Let $\mu\neq 0$ and $a>0$. For every $x,y> a$ and $t>0$ we have
   \formula[eq:mainthm]{
      p_a^{(\mu)}(t,x,y)\stackrel{\mu}{\approx} \left[1\wedge \frac{(x-a)(y-a)}{t}\right]\left(1\wedge \frac{xy}{t}\right)^{|\mu|-\frac{1}{2}} \left(\frac{y}{x}\right)^{\mu+\frac{1}{2}}\frac{1}{\sqrt{t}}\exp\left(-\frac{(x-y)^2}{2t}\right)\/.
      }
\end{theorem}
   Here $f(t,x,y)\stackrel{\mu}{\approx} g(t,x,y)$ means that there exist positive constants $c_1$ and $c_2$ depending only on the index $\mu$ such that $c_1\leq f/g\leq c_2$ for every $x,y>a$ and $t>0$. Since the constants are independent of $a>0$, one can pass to the limit with $a\to 0^+$ and obtain the well-known estimates of $p^{(\mu)}(t,x,y)$. Since the function $I_\mu(z)$ behaves as a power function at zero and that some exponential term appears in the asymptotic expansion at infinity (see Preliminaries for the details), the behaviour of $p^{(\mu)}(t,x,y)$ depends on the ratio $xy/t$. Note that similar situation takes place in the case of $p_a^{(\mu)}(t,x,y)$, which depends  on $xy/t$ as well. It can be especially seen in the proof of Theorem \ref{thm:main}, where different methods and arguments are applied to obtain estimates  (\ref{eq:mainthm}), whenever $xy/t$ is large or small. Finally, taking into account the behaviour of $p^{(\mu)}(t,x,y)$, one can rewrite the statement of Theorem \ref{thm:main} in the following way
   \formula[eq:mainthm:rewrite]{
      \frac{p_a^{(\mu)}(t,x,y)}{p^{(\mu)}(t,x,y)} \stackrel{\mu}{\approx} \left(1\wedge \frac{(x-a)(y-a)}{t}\right)\left(1\vee \frac{t}{xy}\right)\/,\quad x,y>a\/,\quad t>0\/,
     }
  where the expression on the right-hand side of (\ref{eq:mainthm:rewrite}) should be read as the description of the behaviour of $p_a^{(\mu)}(t,x,y)$ near the boundary $a$.
   
   There are several ways to define the function $p_a^{(\mu)}(t,x,y)$ hence our result and its applications can be considered from different points of view. It seems to be the most classical approach to define the heat kernel $p_a^{(\mu)}(t,x,y)$ as the fundamental solution of the heat equation $
      \left(\partial_t-L^{(\mu)}\right)u=0$,   where $L^{(\mu)}$ is the Bessel differential operator. In the most classical case, i.e. when the operator $L^{(\mu)}$ is replaced by the classical Laplacian, the problem of finding description of the heat kernel has a very long history (see for example \cite{SC:2010} and the references within) and goes back to 1980s and the works of E.B. Davies (see \cite{DaviesSimon:1984}, \cite{Davies:1987}, \cite{Davies:1990}, \cite{Davies:1991}). However, the known results for Dirichlet Laplacian on the subsets of $\R^n$ (see \cite{Zhang:2002}) or in general on Riemannian manifolds (see \cite{SC:2010} for the references) are only qualitatively sharp, i.e. the constants appearing in the exponential terms in the upper and  lower estimates are different. Note that in our result these constants are the same and consequently, the exponential behaviour of the density is very precise. Such sharp estimates seems to be very rare.
   
   Note also that the operator $L^{(\mu)}$ plays an important r\^o{}le in harmonic analysis. However, since the set $(a,\infty)$ is unbounded, our consideration corresponds to the case when the spectrum is continuous. This operator on the set $(0,1)$ and the estimates of the corresponding Fourier-Bessel heat kernel were studied recently in \cite{NowakRoncal:2013a} and \cite{NowakRoncal:2013b}, but once again the results presented there are only qualitatively sharp, i.e. the estimates are not sharp whenever $|x-y|^2>>t$. Another essential difference between the case of bounded sets and our case is that in the first one, we can limit our considerations to $t\leq 1$, by the application of the intrinsic ultracontractivity. 
	However, the most  interesting part of Theorem \ref{thm:main} (with difficult proof) seems to be when $t$ is large.
   
   The third and our principal motivation comes from the theory of stochastic processes and the interpretation of $p_a^{(\mu)}(t,x,y)$ as a transition density function of the killed semi-group related to the Bessel process $R_t^{(\mu)}$. From this point of view, the present work is a natural continuation of the research started in \cite{BR:2006} (see also \cite{BGS:2007}), where the integral representation of the density 
	$q_x^{(\mu)}(t)$ of $T_a^{(\mu)}$ were provided together with its some asymptotics description. The sharp estimates of the density for the whole range of parameters with the explicit description of the exponential behaviour was given in \cite{BMR3:2013}. For the in-depth analysis of the asymptotic behaviour of $q_x^{(\mu)}(t)$ see \cite{HM:2013a}, \cite{HM:2012}, \cite{HM:2013b}.
   
  The case $\mu=0$ is excluded from our consideration and it will be addressed in the subsequent work. As it is very common in this theory, this case requires different methods and should be considered separately. In particular, some logarithmic behavior is expected whenever $xy<t$.

  The paper is organized as follows. In Preliminaries we introduce some basic notation and recall properties and known results related to modified Bessel functions as well as Bessel processes, which are used in the sequel. In particular, using scaling property and absolute continuity of the Bessel processes we reduced our consideration only to the case $\mu>0$ and $a=1$. After that we turn to the proof of Theorem \ref{thm:main}, which is split into two main parts, i.e. in Section \ref{section:xyt:large} we provide estimates whenever $xy/t$ is large and in Section \ref{section:xyt:small} we prove (\ref{eq:mainthm}) for $xy/t$ small. In both cases the result is given in series of propositions.

\section{Preliminaries}
\label{section:preliminaries}
\subsection{Notation}
  The constants depending on the index $\mu$ and appearing in theorems and propositions are denoted by capitals letters $C_1^{(\mu)}, C_2^{(\mu)},\ldots$. We will denote by $c_1,c_2,\ldots$ constants appearing in the proofs and to shorten the notation we will omit the superscript $\,^{(\mu)}$, however we will emphasize the dependence on the other variables, if such occurs. 
  
\subsection{Modified Bessel function}
\textit{The modified Bessel function of the first kind} is defined as (see \cite{Erdelyi:1953:volII} 7.2.2 (12))
\formula{
  I_\mu(z) = \sum_{k=0}^\infty \left(\frac{z}{2}\right)^{\mu+2k}\frac{1}{k!\Gamma(k+\mu+1)}\/,\quad z>0\/,\quad \mu>-1\/.
}
It is well-known that whenever $z$ is real the function is a positive increasing real function. Moreover, by the differentiation formula (see \cite{Erdelyi:1953:volII} 7.11 (20))
\formula[eq:I:diff]{
   \dfrac{d}{dz}\left(\frac{I_\mu(z)}{z^\mu}\right) = \frac{I_{\mu+1}(z)}{z^\mu}\/,\quad z>0\/
}
and positivity of the right-hand side of (\ref{eq:I:diff}) we obtain that $z \to z^{-\mu}I_\mu(z)$ is  also increasing.

 The  asymptotic behavior of $I_\mu(z)$ at zero follows immediately from the series representation of $I_\mu(z)$
\formula[eq:I:asym:zero]{
  I_\mu(z) &= \left(\frac{z}{2}\right)^\mu \frac{1}{\Gamma(\mu+1)}+O(z^{\mu+2})\/,\quad z\to 0^{+}\/,
  }
  where the behaviour at infinity is given by (see \cite{Erdelyi:1953:volII} 7.13.1 (5))
\formula[eq:I:asym:infty]{
  I_\mu(z) &\sim \frac{e^z}{\sqrt{2\pi z}}\left(1+O(1/z)\right)\/,\quad z\to\infty\/.
}
Some parts of the proof strongly depends on the estimates of the ratio of two modified Bessel functions with different arguments. Here we recall the results of Laforgia given in Theorem 2.1 in \cite{Laforgia:1991}. For every $\mu>-1/2$ we have 
\formula[MBF:ineq:upper]{
\frac{I_\mu(y)}{I_\mu(x)}<\left(\frac{y}{x}\right)^{\mu}e^{y-x}\/,\quad y\geq x>0\/.
}
Moreover, whenever $\mu\geq1/2$, the lower bound of similar type holds, i.e. we have
\formula[MBF:ineq:lower]{
\frac{I_\mu(y)}{I_\mu(x)}\geq \left(\frac{x}{y}\right)^{\mu}e^{y-x}\/,\quad y\geq x>0\/.
}

\subsection{Bessel processes} 
In this section we introduce basic properties of Bessel processes. We follow the notation presented in \cite{MatsumotoYor:2005a} and \cite{MatsumotoYor:2005b}, where we refer the reader for more details. 

We write $\pr_x^{(\mu)}$ and $\ex^{(\mu)}_x$ for the probability law and the corresponding expected value of a Bessel process $R_t^{(\mu)}$ with an index $\mu\in\R$ on the canonical path space with starting point $R_0=x>0$. The filtration of the coordinate process is denoted by $\mathcal{F}_t^{(\mu)}=\sigma\{R_s^{(\mu)}:s\leq t\}$. The laws of Bessel processes with different indices are absolutely continuous and the corresponding Radon-Nikodym derivative is described by
\formula[ac:formula]{
\left.\frac{d\pr^{(\mu)}_x}{d\pr^{(\nu)}_x}\right|_{\mathcal{F}_t}=\left(\frac{w(t)}{x}\right)^{\mu-\nu}\exp\left(-\frac{\mu^{2}-\nu^2}{2}\int_{0}^{t}\frac{ds}{w^{2}(s)}\right)\/,
}
where $x>0$, $\mu,\nu\in\R$ and the above given formula holds $\pr^{(\nu)}_x$-a.s on $\{T_0^{(\nu)}>t\}$. Here $T_0^{(\mu)}$ denotes the first hitting time of $0$ by $R_t^{(\mu)}$. The behaviour of $R_t^{(\mu)}$ at zero depends on $\mu$. Since we are interested in a Bessel process in a half-line $(a,\infty)$, for a given strictly positive $a$, the boundary condition at zero is irrelevant from our point of view. However, for completeness of the exposure we impose killing condition at zero for $-1<\mu<0$, i.e. in the situation when $0$ is non-singular. Then the density  of the transition probability (with respect to the Lebesgue measure) is given by (\ref{eq:transitiondensity:formula}). 

For $x> 0$ we define the first hitting of a given level $a>0$ by 
\formula{
   T_a^{(\mu)}  =\inf\{t>0: R_t^{(\mu)}=a\}\/.
}
Notice that for $\mu\geq 0$ we have $T_a^{(\mu)}<\infty$ a.s., but for $\mu<0$ the variable $T_a^{(\mu)}$ is infinite with positive probability.
We denote by $q_{x,a}^{(\mu)}(s)$ the density function of $T_a^{(\mu)}$. The sharp estimates of $q_{x,a}^{(\mu)}(s)$ were obtained in \cite{BMR3:2013}. We recall this result for $a=1$, which  implies the result for every $a>0$, due to the scaling property of Bessel processes.  More precisely, it was shown that for every $x>1$ and $t>0$ we have
  \formula[hittingtime:estimates]{
q_{x,1}^{(\mu)}(s) &\stackrel{\mu}{\approx}
(x-1)\left(1\wedge\frac{1}{x^{2\mu}}\right)\frac{
{e^{-(x-1)^2/(2t)}}}{t^{3/2}} \frac{ x^{2|\mu|-1} }{t^{|\mu|-1/2}+
x^{|\mu|-1/2}}\/,\quad \mu\neq 0\/.
  }
The above-given bounds imply the description of the survival probabilities (see Theorem 10 in \cite{BMR3:2013} )
        \formula[sp:estimate:mu]{
           \textbf{P}_x^{(\mu)}(T^{(\mu)}_1>t) &\stackrel{\mu}{\approx} \frac{x-1}{\sqrt{x\wedge t}+x-1}\frac{1}{t^{\mu}+x^{2\mu}}\/, \quad x>1\/,\quad t>0\/.
        }

The main object of our study is the density of the transitions probabilities for the Bessel process starting from $x>a$ killed at time $T_a^{(\mu)}$. Taking into account the Hunt formula (\ref{eq:hunt:general}) and the fact that continuity of the paths implies $R_{T_a^{(\mu)}}^{(\mu)}=a$ a.s., we can represent $p_a^{(\mu)}(t,x,y)$ in terms of $p^{(\mu)}(t,x,y)$ and $q_{x,a}^{(\mu)}(s)$ in the following way
\formula[eq:hunt:formula]{
p_a^{(\mu)}(t,x,y) & = p^{(\mu)}(t,x,y) - r_a^{(\mu)}(t,x,y)\\
\label{eq:hunt:formula2}
    &= p^{(\mu)}(t,x,y)- \int_0^t p^{(\mu)}(t-s,a,y)q_{x,a}^{(\mu)}(s)ds\/.
}
The scaling property of a Bessel process together with (\ref{eq:hunt:formula2}) imply that
\formula[eq:pt1:scaling]{
  p_a^{(\mu)}(t,x,y) = \frac{1}{a}p_1^{(\mu)}(t/a^2,x/a,y/a)\/,\quad x,y>a\/,\quad t>0\/.
}
Moreover, the absolute continuity property (\ref{ac:formula}) applied for $\mu>0$ and $\nu=-\mu$ gives
\formula{
   p_1^{(-\mu)}(t,x,y) = \left(\frac{x}{y}\right)^{2\mu}p_1^{(\mu)}(t,x,y)\/,\quad x,y>1\/,\quad t>0\/.
}
These two properties show that it is enough to prove Theorem \ref{thm:main} only for $a=1$ and $\mu>0$. To shorten the notation we will write $q_x^{(\mu)}(s)=q_{x,1}^{(\mu)}(s)$. Since we consider the densities with respect to the Lebesgue measure (not with respect to the speed measure $m(dx)=2x^{2\mu+1}dx$) the symmetry property of $p_1^{(\mu)}(t,x,y)$ in this case reads as follows:
\formula[eq:pt1:symmetry]{
   p_1^{(\mu)}(t,x,y) = \left(\frac{y}{x}\right)^{2\mu+1}p_1^{(\mu)}(t,y,x)\/,\quad x,y>1\/,\quad t>0\/.
}
Finally, for $\mu=1/2$ one can compute $p_1^{(\mu)}(t,x,y)$ explicitly from (\ref{eq:hunt:formula2}), by using $I_{1/2}(z)=\sqrt{\frac{2}{{\pi z}}}\sinh(z)$ and the fact that $q_{x}^{(1/2)}(s)$ is a density of $1/2$-stable subordinator. More precisely, since
\formula[formula:pt:12:1]{
q_{x}^{(1/2)}(t)&=\frac{x-1}{x}\frac{1}{\sqrt{2\pi t^3}}\exp{\left(-\frac{(x-1)^2}{2t}\right)},\\
\label{formula:pt:12:2}
p^{(1/2)}(t,x,y)&=\frac{1}{\sqrt{2\pi t}}\frac{y}{x} \left(\exp{\left(-\frac{(x-y)^2}{2t}\right)}-\exp{\left(-\frac{(x+y)^2}{2t}\right)}\right) ,
}
we obtain
\formula{r_{1}^{(1/2)}(t,x,y)=&\int_{0}^{t}q_{x}^{(1/2)}(s)p^{(1/2)}(t-s,1,y)ds\\
=&\frac{x-1}{x}\frac{y}{2\pi}(H(t,(x-1)^2,(y-1)^2)-H(t,(x-1)^2,(y+1)^2)),}
where
\formula{
H(t,a,b)=\int_{0}^{t}\frac{1}{\sqrt{t-s}}\frac{1}{\sqrt{s^3}}\exp{\left(-\frac{a}{2s}\right)}\exp{\left(-\frac{b}{2(t-s)}\right)}ds\/,\quad a,b>0\/.
}
Making the  substitution $w=1/s-1/t$ and using formula 3.471.15 in \cite{GradsteinRyzhik:2007}
%\jm{(dodac odnosnik do Ryzhika-Gradsteina)} 
we get
\formula[eq:H:final]{
\nonumber
H(t,a,b)&=\frac{1}{\sqrt{t}}\exp{\left(-\frac{a + b}{2t}\right)}\int_{0}^{\infty}w^{-1/2}\exp{\left(-\frac{a}{2}w-\frac{b}{2w}\right)}dw\\
&=\sqrt{\frac{2\pi}{ta}}
\exp{\left(-\frac{(\sqrt{a}+\sqrt{b})^2}{2t}\right)}\/.
}
Hence we have
\formula[eq:rt:12:formula]{
r_{1}^{(1/2)}(t,x,y)=\frac{1}{\sqrt{2\pi t}}\frac{y}{x}\left[\exp{\left(-\frac{(x+y-2)^2}{2t}\right)}-\exp{\left(-\frac{(x+y)^2}{2t}\right)}\right]
}
which together with (\ref{eq:hunt:formula2}) and (\ref{formula:pt:12:2}) give
\formula[eq:pt1:12:formula]
{ p_{1}^{(1/2)}(t,x,y)=\frac{1}{\sqrt{2\pi t}} \frac{y}{x}\left(\exp\left(-\frac{(x-y)^2}{2t}\right)-\exp\left(-\frac{(x+y-2)^2}{2t}\right)\right).
}
One can also obtain this formula using the relation between $3$-dimensional Bessel process (i.e. with index $\mu=1/2$) and $1$-dimensional Brownian motion killed when leaving a positive half-line. Note also that 
\formula[eq:pt1:12:asympt]{
   p_1^{(1/2)}(t,x,y) \approx  \left(1\wedge\frac{(x-1)(y-1)}{t}\right) \frac{y}{x} \frac{1}{\sqrt{ t}} \exp\left(-\frac{(x-y)^2}{2t}\right)\/.
}
which is exactly (\ref{eq:mainthm}) for $\mu=1/2$.

We end this section providing very useful relation between densities $q_x^{(\mu)}(t)$ with different indices, which once again follows from the absolute continuity property.
\begin{lemma}
\label{lem:q:nu12mu}
For every $x>1$ and $t>0$ we have
\formula[eq:q:nu12mu]{
  x^{\mu-1/2}q_x^{(\mu)}(t)\leq q_x^{(1/2)}(t)\leq x^{\nu-1/2} q_x^{(\nu)}(t)\/,
}
whenever $\nu\leq 1/2\leq \mu$.
\end{lemma}
\begin{proof}
  The second inequality in (\ref{eq:q:nu12mu}) was given in Lemma 4 in \cite{BMR3:2013}. To deal with the right-hand side of (\ref{eq:q:nu12mu}) we use (\ref{ac:formula}) to obtain for every  $\delta>0$ and $0<\varepsilon \leq\delta^2/2\wedge 1$ 
  \formula[eq:lem:proof01]{
  \nonumber
    x^{\mu-1/2}\ex_x^{(\mu)}[t-\varepsilon\leq T_1^{(\mu)}\leq t] &\leq  \ex_x^{(1/2)}\left[t-\varepsilon\leq T_1^{(1/2)}\leq t;\left({R_t}\right)^{\mu-1/2}\right]  \\
    &\leq  \left({1+\delta}\right)^{\mu-1/2}\ex_x^{(1/2)}[t-\varepsilon\leq T_1^{(1/2)}\leq t]+F_\varepsilon(x,t)\/,
    }
    where, by Strong Markov property 
    \formula{
    F_\varepsilon(x,t) &= \ex_x^{(1/2)}[t-\varepsilon\leq T_1^{(1/2)}\leq t,R_t\geq 1+\delta;\left({R_t}\right)^{\mu-1/2}]\\
    &=\ex_x^{(1/2)}[t-\varepsilon\leq T_1^{(1/2)}\leq t; \ex_1^{(1/2)}[R_{t-T_1^{(1/2)}}\geq 1+\delta;\left({R_{t-T_1^{(1/2)}}}\right)^{\mu-1/2}]]\\
    &=\int_{t-\varepsilon}^t q_x^{(1/2)}(u)\int_{1+\delta}^\infty y^{\mu-1/2}p^{(1/2)}(t-u,1,y)\,dydu\/.
  }
  By (\ref{formula:pt:12:2}), for every $r\in (0,\varepsilon)$ we have
  \formula{
     \int_{1+\delta}^\infty y^{\mu-1/2}p^{(1/2)}(r,1,y)\,dy &\leq \frac{1}{\sqrt{2\pi r}}\int_{1+\delta}^\infty \exp\left(-\frac{(y-1)^2}{2r}\right)y^{\mu+1/2}\,dy\\
     &\leq \frac{1}{\sqrt{2\pi r}}\exp\left(-\frac{\delta^2}{4r}\right)\int_{1+\delta}^\infty \exp\left(-\frac{(y-1)^2}{4}\right)y^{\mu+1/2}\,dy\\
     &\leq \frac{1}{\sqrt{2\pi \varepsilon}}\exp\left(-\frac{\delta^2}{4\varepsilon}\right)\int_{1+\delta}^\infty \exp\left(-\frac{(y-1)^2}{4}\right)y^{\mu+1/2}\,dy\/,
  }
  where the last inequality follows from $\varepsilon\leq \delta^2/2$. It implies that $F_\varepsilon(t,x)/\varepsilon$ vanishes when $\varepsilon$ goes to zero. Consequently, dividing both sides of (\ref{eq:lem:proof01}) by $\varepsilon$ and taking a limit when $\varepsilon \to 0$, we arrive at
  \formula{
     x^{\mu-1/2}q_x^{(\mu)}(t)\leq ({1+\delta})^{\mu-1/2}q_x^{(1/2)}(t)\/.
  }
  Since $\delta$ was arbitrary, the proof is complete.
\end{proof}
%%%%%%%%%%%%%%%%%%%%%%%%%%%%%%%%%%%%%%%%%%%

%						xy/t large

%%%%%%%%%%%%%%%%%%%%%%%%%%%%%%%%%%%%%%%%%%%

\section{Estimates for $xy/t$ large}
\label{section:xyt:large}

We begin this Section with the application of the absolute continuity property of Bessel processes and the formula (\ref{eq:pt1:12:formula}) which give the upper bounds for $\mu\geq 1/2$ and lower bounds for $\nu\leq 1/2$. These bounds are sharp whenever $xy\geq t$.
\begin{proposition}
\label{prop:upperbounds:xytlarge}
   Let $\mu\geq 1/2\geq \nu>0$. For every $x,y>1$ and $t>0$ we have
   \formula[eq:up:low]{
       \left(\frac{x}{y}\right)^{\mu-\frac{1}{2}}p_1^{(\mu)}(t,x,y)\leq p_1^{(1/2)}(t,x,y)\leq \left(\frac{x}{y}\right)^{\nu-\frac{1}{2}}p_1^{(\nu)}(t,x,y)\/.
   }
\end{proposition}
\begin{proof}
From the absolute continuity property (\ref{ac:formula}) we get that for every $\mu\geq\nu>0$ and every Borel set $A\subset(1,\infty)$ we have
\formula{
\int_{A}p_{1}^{(\mu)}(t,x,y)dy &=\frac{1}{x^{\mu-\nu}}\ex_{x}^{(\nu)}\left[T_1^{(\nu)}>t,R_t \in A;
(R_t)^{\mu-\nu}\exp\left(-\frac{\mu^{2}-\nu^2}{2}\int_{0}^{t}\frac{ds}{R^{2}_s}\right)\right]\\
&\leq\frac{1}{x^{\mu -\nu}}\ex_{x}^{(\nu)}[T_1^{(\nu)}>t,R_t \in
A;(R_t)^{\mu-\nu}]= \int_A \left(\frac{y}{x}\right)^{\mu -\nu}p_1^{(\nu)}(t,x,y)\,dy\/.}
Hence
\formula[eq:munu:relation]{
p_{1}^{(\mu)}(t,x,y)\leq&\left(\frac{y}{x}\right)^{\mu-\nu}p_{1}^{(\nu)}(t,x,y)\/.
}
Taking $\mu\geq 1/2$ and $\nu=1/2$ gives the left-hand side of (\ref{eq:up:low}) and taking $\nu\leq 1/2$ and $\mu=1/2$ gives the right-hand side of (\ref{eq:up:low}).
\end{proof}
The absolute continuity can also be used to show the estimates for small times $t$ in a very similar way. Note that if $t<1$ then we always have $xy>t$. The proof of the main Theorem will be provided in subsequent propositions without the assumption that $t$ is bounded, but we present this simple proof to show that for $xy\geq t$ the estimates for small $t$ are just an immediate consequence of the absolute continuity of Bessel processes.
\begin{proposition}
% Proposition 
   Let $\mu >0$. For every $x,y> 1$ and $t\in(0,1]$ we have
   \formula[eq:t:small]{
   p_{1}^{(\mu)}(t,x,y)\stackrel{\mu}{\approx}& \left(1\wedge\frac{(x-1)(y-1)}{t}\right)\left(\frac{y}{x}\right)^{\mu+1/2} \frac{1}{\sqrt{t}}\exp\left(-\frac{(x-y)^2}{2t}\right) \/.
   }
\end{proposition}
\begin{proof}
Let $\mu\geq \nu>0$. Taking Borel set $A\subset(1,\infty)$ and $t\leq 1$ we have
\formula{\int_{A}p_{1}^{(\mu)}(t;x,y)dy
=&\frac{1}{x^{\mu-\nu}}\ex_{x}^{(\nu)}\left[T_1^{(\nu)}>t;R_t \in A;
(R_t)^{\mu-\nu}\exp\left(-\frac{\mu^{2}-\nu^2}{2}\int_{0}^{t}\frac{ds}{R^{2}_s}\right)\right]\/.
}
Since $\inf\{R_s:s<t\}>1$ on $\{T_1^{(\nu)}>t\}$ we can write
\formula{
\int_{A}p_{1}^{(\mu)}(t,x,y)dy
\geq&\frac{1}{x^{\mu-\nu}}\ex_{x}^{(\nu)}\left[T_1^{(\nu)}>t;R_t \in A;
(R_t)^{\mu-\nu}\exp\left(-\frac{\mu^{2}-\nu^2}{2}t\right)\right]\\
\geq&\exp\left(-\frac{\mu^{2}-\nu^2}{2}\right)\int_{A}\left(\frac{y}{x}\right)^{\mu-\nu}p_{1}^{(\nu)}(t,x,y)dy\/.}
Hence we get
\formula{p_{1}^{(\mu)}(t,x,y)\geq&\exp\left(-\frac{\mu^{2}-\nu^2}{2}\right)\left(\frac{y}{x}\right)^{\mu-\nu}p_{1}^{(\nu)}(t,x,y)\/.
}
Now taking $\mu\geq 1/2$ and $\nu=1/2$ together with (\ref{eq:pt1:12:formula}) and the result of Proposition \ref{prop:upperbounds:xytlarge} gives the proof of (\ref{eq:t:small}) for $\mu\geq 1/2$. Analogous argument applied for $\mu<1/2$ ends the proof.
\end{proof}
Next proposition together with Proposition \ref{prop:upperbounds:xytlarge} provide the estimates for $x,y$ bounded away from $1$. Notice that if $x,y>c>1$ and $xy>t$ then 
\formula[eq:xy:away:bounds1]{
    \frac{(x-1)(y-1)}{t}\geq  \left(1-\frac{1}{c}\right)^2 \frac{xy}{t} \geq  \left(1-\frac{1}{c}\right)^2 \/.
}
and consequently the right-hand side of (\ref{eq:mainthm:rewrite}) is comparable with a constant which means that $p_1^{(\mu)}(t,x,y)$ is comparable with $p^{(\mu)}(t,x,y)$.
\begin{proposition}
\label{prop:xy:away}
Let $\mu\geq 1/2\geq \nu>0$. Then there exist constants $C_1^{(\nu)},C_2^{(\mu)}>0$ and $C_3^{(\mu)}>1$ such that 
\formula{
   C_1^{(\nu)}\left(\frac{x}{y}\right)^{\nu+1/2}p_1^{(\nu)}(t,x,y)\leq  \frac{1}{\sqrt{t}}\exp\left(-\frac{(x-y)^2}{t}\right)\leq  C_2^{(\mu)}\left(\frac{x}{y}\right)^{\mu+1/2}p_1^{(\mu)}(t,x,y)\/, 
}
whenever $xy\geq t$ and the lower bounds holds for $x,y>2$ and the upper bounds are valid for $x,y>C_3^{(\mu)}$.
\end{proposition}
\begin{proof}
Taking $0<\nu\leq 1/2$ and using the description of the behaviour of $I_\nu(z)$ at infinity \eqref{eq:I:asym:infty} together with general estimate $p_1^{(\nu)}(t,x,y)\leq p^{(\nu)}(t,x,y)$ (which is an immediate consequence of the definition (\ref{eq:hunt:general})) we get
\formula{
p_1^{(\nu)}(t,x,y)&\leq p^{(\nu)}(t,x,y)
\stackrel{\nu}{\approx}\frac{1}{\sqrt{t}}\left(\frac{y}{x}\right)^{\nu+1/2}\exp\left(-\frac{(x-y)^2}{t}\right)\/.
}
This ends the proof for small indices. 

Now let $\mu\geq 1/2$. Since the modified Bessel function $I_\mu(z)$ is positive, continuous and behaves like $(2\pi z)^{-1/2}e^{z}$ at infinity (see (\ref{eq:I:asym:infty})) there exists constant $c_1>1$ such that 
   \formula{
      I_\mu\left(\frac{xy}{t}\right)\geq \frac{1}{c_1}\sqrt{\frac{t}{2\pi xy}}\exp\left(\frac{xy}{t}\right)\/,
   }
   whenever $xy\geq t$. One can show that it is enough to take $c_1 = (I_\mu(1)e^{-1}\sqrt{2\pi})^{-1}$.
Consequently, applying above given estimate to (\ref{eq:transitiondensity:formula}) we arrive at
   \formula[eq:proof:1]{
      \left(\frac{y}{x}\right)^{\mu-1/2}p^{(1/2)}(t,x,y) \geq p^{(\mu)}(t,x,y)\geq \frac{1}{c_1}\frac{1}{\sqrt{2\pi t}}\left(\frac{y}{x}\right)^{\mu+1/2}\exp\left(-\frac{(x-y)^2}{2t}\right)\/,\quad {xy}\geq{t}\/,
   }
   where the first inequality is just (\ref{eq:munu:relation}). Moreover, by (\ref{eq:q:nu12mu}), we have
   \formula{
     q_x^{(\mu)}(t)\leq \frac{q_x^{(1/2)}(t)}{x^{\mu-1/2}} = \frac{x-1}{x^{\mu-1/2}}\frac{1}{\sqrt{2\pi}t^{3/2}}\exp\left(-\frac{(x-1)^2}{2t}\right)\/,\quad t>0,x>1\/.
   }
   and it together with left-hand side of (\ref{eq:proof:1}) and (\ref{eq:rt:12:formula}) imply
   \formula{
      {r_1^{(\mu)}(t,x,y)} &= \int_0^t q_x^{(\mu)}(s)p^{(\mu)}(t-s,1,y)\,ds
      \leq  \left(\frac{y}{x}\right)^{\mu-1/2} \int_0^t q_x^{(1/2)}(s)p^{(1/2)}(t-s,1,y)\,ds\\
      &= \left(\frac{y}{x}\right)^{\mu-1/2} r_1^{(1/2)}(t,x,y)\\
      & = \frac{1}{\sqrt{2\pi t}}\left(\frac{y}{x}\right)^{\mu+1/2}\left(\exp\left(-\frac{(x+y-2)^2}{2t}\right)-\exp\left(-\frac{(x+y)^2}{2t}\right)\right)\/.
   }
   Let $C_3^{(\mu)} = \left(1-\sqrt{\frac{2c_1}{2c_1+1}}\right)^{-1}$ and taking into account right-hand side of (\ref{eq:proof:1}) and (\ref{eq:xy:away:bounds1}) we obtain for $x,y>C_3^{(\mu)}$ that
   \formula{
     \frac{r_1^{(\mu)}(t,x,y)}{p^{(\mu)}(t,x,y)}&\leq c_1\exp\left(\frac{(x-y)^2}{2t}\right)\left(\exp\left(-\frac{(x+y-2)^2}{2t}\right)-\exp\left(-\frac{(x+y)^2}{2t}\right)\right)\\
     &=c_1\left(\exp\left(-\frac{2(x-1)(y-1)}{t}\right)-\exp\left(-\frac{2xy}{t}\right)\right)\\
     &\leq c_1\left(\exp\left(-c_2\frac{2xy}{t}\right)-\exp\left(-\frac{2xy}{t}\right)\right)\/,
   }
   where 
   \formula{
   c_2 = \left(1-\frac{1}{C_3^{(\mu)}}\right)^2 = \frac{2c_1}{2c_1+1}<1\/.
   }
   Taking into account the general estimate 
   \formula{
      e^{-c_2z}-e^{-z}\leq \frac{1-c_2}{c_2}\/,\quad z>0\/,c_2<1
   }
   we arrive at
   \formula{
   \frac{r_1^{(\mu)}(t,x,y)}{p^{(\mu)}(t,x,y)}&\leq c_1\frac{1-c_2}{c_2} = \frac{1}{2}\/.
   }
Consequently
\formula{
   p_1^{(\mu)}(t,x,y)\geq \frac12 p^{(\mu)}(t,x,y)\geq \frac{1}{2c_1} \frac{1}{\sqrt{2\pi t}}\left(\frac{y}{x}\right)^{\mu+1/2}\exp\left(-\frac{(x-y)^2}{2t}\right)\/.
}
\end{proof}

Now we turn our attention to the case when $x$ and $y$ are bounded. The next proposition, however, is much more general.
\begin{proposition}
\label{prop:xminy}
% Proposition 
For fixed $m>0$ and $\mu\geq 1/2\geq \nu>0$ there exist constants $C^{(\mu)}_4,C^{(\nu)}_4>0$ such that 
\formula{
C^{(\mu)}_4\left(\frac{x}{y}\right)^{\mu+1/2}p_{1}^{(\mu)}(t,x,y)&\geq \left(1\wedge\frac{(x-1)(y-1)}{t}\right)\frac{1}{\sqrt{
t}}\exp\left(-\frac{(x-y)^{2}}{2t}\right)}
and
\formula{
\left(1\wedge\frac{(x-1)(y-1)}{t}\right)\frac{1}{\sqrt{
t}}\exp\left(-\frac{(x-y)^{2}}{2t}\right)\geq C^{(\nu)}_4\left(\frac{x}{y}\right)^{\nu+1/2}p_{1}^{(\nu)}(t,x,y)
}
whenever ${(x \wedge y)^{2}}\geq mt$.
\end{proposition}

\begin{proof} 
Without lost of generality we can assume that $1<x<y$. We put $b=(x+1)/2$ and take $\mu\geq 1/2$. Using (\ref{ac:formula}) and the fact that $T_b^{(1/2)}\leq T_1^{(1/2)}$ we can write for every Borel set $A\subset (1,\infty)$ that
\formula{
\int_{A}p_{1}^{(\mu)}(t,x,y)dy&
\geq \ex_{x}^{(1/2)}\left[t<T_b^{(1/2)},R_t \in
A;\left(\frac{R_{t}}{x}\right)^{\mu-1/2}\exp\left(-\frac{\mu^2 -1/4}{2}\int_{0}^{t}\frac{ds}{R_s^{2}}\right)\right]
}
Since up to time $T_b^{(1/2)}$ we have 
\formula{
\int_{0}^{t}\frac{ds}{R_s^{2}}\leq \frac{4t}{(x+1)^2}\leq \frac{4t}{x^2}\leq \frac{4}{m}\/,
}
we obtain
\formula{
\int_{A}p_{1}^{(\mu)}(t,x,y)dy\geq \exp\left(-\frac{4\mu^2-1}{2m}\right)\ex_{x}^{(1/2)}\left[t<T_b^{(1/2)},R_t
\in A;\left(\frac{R_t}{x}\right)^{\mu-1/2}\right]\/,
}
which gives 
\formula[eq:proof:2]{
p_1^{(\mu)}(t,x,y)\geq \exp\left(-\frac{4\mu^2-1}{2m}\right)\left(\frac{y}{x}\right)^{\mu-1/2}p_b^{(1/2)}(t,x,y)\/.
}
From the other side, the scaling property \eqref{eq:pt1:scaling} and the formula \eqref{eq:pt1:12:asympt} give
\formula{
p_{b}^{(1/2)}(t,x,y) &=\frac{1}{b}p_{1}^{(1/2)}\left(\frac{t}{b^2};\frac{x}{b},\frac{y}{b}\right)\\
&\approx \frac{1}{\sqrt{t}}\frac{y}{x}\exp\left(-\frac{(x-y)^2}{2t}\right)\left(1\wedge \frac{(x-b)(y-b)}{t}\right)\\
&\approx \frac{1}{\sqrt{t}}\frac{y}{x}\exp\left(-\frac{(x-y)^2}{2t}\right)\left(1\wedge \frac{(x-1)(y-1)}{t}\right)\/,
}
where the last equalities follows from
\formula{
   x-b = \frac{x-1}{2}\/,\quad \frac{y-1}{2}\leq y-b\leq y-1\/.
}
It ends the proof for $\mu\geq 1/2$.

For $1/2\geq \nu>0$ we similarly write 
\formula{
\int_{A}p_{b}^{(1/2)}(t,x,y)dy&
\leq \ex_{x}^{(\nu)}\left[t<T_1^{(\nu)},R_t \in
A;\left(\frac{R_{t}}{x}\right)^{1/2-\nu}\exp\left(\frac{\nu^2-1/4}{2}\int_{0}^{t}\frac{ds}{R_s^{2}}\right)\right]
}
and we obtain 
\formula{
p_{b}^{(1/2)}(t,x,y)\leq \exp\left(\frac{4\nu^2-1}{2m}\right)\left(\frac{y}{x}\right)^{1/2-\nu}p_1^{(\nu)}(t,x,y)\/.
}
This together with the above-given estimates for $p_b^{(1/2)}(t,x,y)$ finish the proof. 
\end{proof}

Since for $x,y<C$ and $xy\geq t$, for some fixed $C>1$, we have
\formula{
   \frac{(x\wedge y)^2}{t}\geq \frac{xy}{Ct}\geq \frac{1}{C}\/,
}
applying the results of Proposition \ref{prop:xminy} (with $m=C^{-1}$) and Proposition \ref{prop:upperbounds:xytlarge} gives

\begin{corollary}
\label{Cor:xy:bounded}
For every $C>1$ we have
\formula{
  p_{1}^{(\mu)}(t,x,y)\stackrel{\mu, C}{\approx}& \left(1\wedge\frac{(x-1)(y-1)}{t}\right)\left(\frac{y}{x}\right)^{\mu+1/2} \frac{1}{\sqrt{t}}\exp\left(-\frac{(x-y)^2}{2t}\right)
}
whenever $x,y<C$ and $xy\geq t$.
\end{corollary}

Finally, we end this section with two propositions related to the case when one of the space variables is close to $1$ and the other is large. We deal with this case separately for $\mu<1/2$ and $\mu\geq 1/2$.

\begin{proposition}
For every $\nu\in (0,1/2)$ there exists constant $C_5^{(\nu)}>0$ such that 
\formula{
p_1^{(\nu)}(t,x,y)\leq C_5^{(\nu)} \frac{1}{\sqrt{t}}\left(\frac{y}{x}\right)^{\nu+1/2}\exp\left(-\frac{(x-y)^2}{2t}\right)\left(1\wedge \frac{(x-1)(y-1)}{t}\right)
}
for $1<x\leq 2\leq y$ and $xy\geq t$.
\end{proposition}
\begin{proof}
By monotonicity of $I_\nu(z)$, for every $s\in(0,t)$ we have
\formula{
\frac{1}{t-s}\geq\frac{1}{\sqrt{t}}\frac{1}{\sqrt{t-s}}, \quad I_{\nu}\left(\frac{y}{t-s}\right)\geq I_{\nu}\left(\frac{y}{t}\right)\/.
}
Hence, using the right-hand side of (\ref{eq:q:nu12mu}), we get
\formula{
q_x^{(\nu)}(s)\geq \frac{x-1}{\sqrt{2\pi s^3}}\frac{1}{x^{\nu+1/2}}\exp{\left(-\frac{(x-1)^2}{2s}\right)}, \quad 0\leq \nu <1/2, \ s>0\/,
}
and the formula (\ref{eq:transitiondensity:formula}) we get
\formula{
r_1^{(\nu)}(t,x,y)&=\int_0^{t}q_x^{(\nu)}(s)\frac{y^{1+\nu}}{t-s}\exp{\left(-\frac{1+y^2}{2(t-s)}\right)}I_{\nu}\left(\frac{y}{t-s}\right)\,ds\\
&\geq\frac{x-1}{\sqrt{2\pi }}\left(\frac{y}{x}\right)^{\nu+1}\sqrt{\frac{x}{t}}I_{\nu}\left(\frac{y}{t}\right)H(t,(x-1)^2,1+y^2)\\
%\int_{0}^{t}\frac{1}{\sqrt{s^3}}\frac{1}{\sqrt{t-s}}\exp{\left(-\frac{(x-1)^2}{2s}\right)}\exp{\left(-\frac{1+y^2}{2(t-s)}\right)}\,ds\\
&={\frac{\sqrt{x}}{t}}\left(\frac{y}{x}\right)^{\nu+1}I_{\nu}\left(\frac{y}{t}\right)\exp{\left(-\frac{(x-1+\sqrt{y^2 +1})^2}{2t}\right)}
}
where the last equality follows from (\ref{eq:H:final}). Using \eqref{MBF:ineq:upper} we obtain
\formula{
p^{(\nu)}(t,x,y) &= \frac{y^{\nu+1}}{t}\exp\left(-\frac{x^2+y^2}{2t}\right)\,\frac{1}{x^{\nu}}I_\nu\left(\frac{xy}{t}\right)\\
&\leq \frac{y^{\nu+1}}{t}\exp\left(-\frac{(x-y)^2}{2t}\right)\exp\left(-\frac{y}{t}\right)I_\nu\left(\frac{y}{t}\right)\/,
}
which together with previously given estimates, \eqref{eq:hunt:formula} and finally \eqref{eq:I:asym:infty} give
\formula{
p_1^{(\nu)}(t,x,y) &\leq \frac{y^{\mu+1}}{t}\exp{\left(-\frac{(x-y)^2}{2t}\right)}\exp{\left(-\frac{y}{t}\right)}I_{\nu}\left(\frac{y}{t}\right)f_{y,t}(x)\\
&\leq c_1 \frac{y^{\mu+1/2}}{\sqrt{t}}\exp{\left(-\frac{(x-y)^2}{2t}\right)}f_{y,t}(x)\/,
}
where 
\formula{
f_{y,t}(x) = 1-\frac{1}{x^{\nu+1/2}}\exp\left(-\frac{(x-1)(\sqrt{y^2+1}+y-1)}{t}\right)\/.
}
By elementary computation we can see that 
\formula{
-f_{y,t}'(x) &=  \frac{1}{x^{\nu+3/2}}\exp\left(-\frac{(x-1)(\sqrt{y^2+1}+y-1)}{t}\right)\left(\frac{\sqrt{y^2+1}+y-1}{t}x+\nu+1/2\right)\\
&\leq \frac{1}{x^{\nu+3/2}}\left(\frac{2xy}{t}+1\right)\leq \frac{4xy}{t} \leq 16\frac{y-1}{t}\/.
}
Here we have used the following inequalities
\formula{
  \sqrt{y^2 +1}+y-1<2y\/,\quad  xy\geq t\/,\quad 1<x\leq 2\leq y\/.
}
Thus, by the mean value theorem, there exists $d=d_{x,y,t}\in (1,x)$ such that 
\formula{
f_{y,t}(x)&= (1-x)f_{y,t}'(d)\leq 16\frac{(x-1)(y-1)}{t}\/.
}
\formula{
p_1^{(\nu)}(t,x,y)&\leq c_1 2^{\nu+9/2} \left(1\wedge \frac{(x-1)(y-1)}{t}\right) \left(\frac{y}{x}\right)^{\nu+1/2}\frac{1}{\sqrt{t}}\exp\left(-\frac{(x-y)^2}{2t}\right)\/.
}
\end{proof}

\begin{proposition}
% Proposition
\label{prop:xyt:large:lower:3}
For every $\mu\geq 1/2$ and $c>1$ there exists constant $C_6^{(\mu)}(c)>0$ such that for every  $1<x\leq c$ and $y\geq 5c(\mu+1)$ we have
\formula{
p_1^{(\mu)}(t,x,y)\geq C_6^{(\mu)}(c)\frac{1}{\sqrt{t}}\left(\frac{y}{x}\right)^{\mu+1/2}\exp{\left(-\frac{(x-y)^2}{2t}\right)}\left(1\wedge \frac{(x-1)(y-1)}{t}\right)\/,
}
whenever $xy\geq t$.
\end{proposition}
\begin{proof}
Let us fix $\mu\geq 1/2$. For every $0<s<t$, using \eqref{MBF:ineq:upper}, we have
\formula{
I_{\mu}\left(\frac{y}{t-s}\right)<I_{\mu}\left(\frac{y}{t}\right)\left(\frac{t}{t-s}\right)^{\mu}\exp\left({\frac{y}{t-s}}\right)\exp\left({-\frac{y}{t}}\right)
}
and consequently
\formula{
\frac{p^{(\mu)}(t-s,1,y)}{p^{(\mu)}(t,1,y)}&<\left(\frac{t}{t-s}\right)^{\mu +1}\exp{\left(-\frac{(y-1)^2}{2}\left(\frac{1}{t-s}-\frac{1}{t}\right)\right)}
=\frac{g_y(t-s)}{g_y(t)}\/,
}
where 
\formula{
g_y(w)=\left(\frac{1}{w}\right)^{\mu +1}\exp\left({-\frac{(y-1)^2}{2w}}\right)\/,\quad w>0\/.
}
Note that 
\formula{
g_y'(w) = \left(\frac{1}{w}\right)^{\mu +2}\exp\left({-\frac{(y-1)^2}{2w}}\right)\left(\frac{(y-1)^2}{2w}-(\mu +1)\right)\/.
}
Since $x\leq c$, $y\geq 5c(\mu+1)>2$ and $xy\geq t$ we have $4(y-1)\geq 2y \geq 2t/c$. Moreover $y-1\geq 4c(\mu+1)$. Thus
\formula{
\frac{(y-1)^2}{2t}\geq \frac{4c(\mu+1)(y-1)}{2t}\geq \mu+1\/.
}
It means that under our assumptions on $x$, $y$ and $t$ the function $g_y(w)$ is increasing on $(0,t)$ and consequently $g_y(t-s)\leq g_y(t)$ for every $0<s<t$.
\formula{
r_{1}^{(\mu)}(t,x,y) &=\int_{0}^{t}q_{x}^{(\mu)}(s)p^{(\mu)}(t-s,1,y)ds\leq p^{(\mu)}(t,1,y)\int_{0}^{t}q_{x}^{(\mu)}(s)ds\\
&= x^{-\mu}\exp{\left(\frac{x^2-1}{2t}\right)}\frac{I_{\mu}\left(y/t\right)}{I_{\mu}\left(xy/t\right)}p^{(\mu)}(t,x,y)\/.
}
The above-given ratio of modified Bessel functions can be estimated from above by using \eqref{MBF:ineq:lower} as follows
\formula{
I_{\mu}\left(\frac{y}{t}\right)\leq I_{\mu}\left(\frac{xy}{t}\right)\exp{\left(-\frac{(x-1)y}{t}\right)}x^{\mu}\/.
} 
Consequently
\formula{r_{1}^{(\mu)}(t,x,y)\leq p^{(\mu)}(t,x,y)\exp{\left(-\frac{(x-1)(2y-x-1)}{2t}\right)}.}
Finally observe that $2y-x-1>y-1$ and we arrive at
\formula{
p_{1}^{(\mu)}(t,x,y)&\geq \left(1-\exp{\left(-\frac{(x-1)(y-1)}{2t}\right)}\right)p^{(\mu)}(t,x,y)\\
%&=\left(1-\exp{\left(-\frac{(x-1)(y-1)}{2t}\right)}\right)\frac{1}{\sqrt{t}}\left(\frac{y}{x}\right)^{\mu+1/2}\sqrt{\frac{xy}{t}}I_{\mu}\left(\frac{xy}{t}\right)\exp{\left(-\frac{x^2+y^2}{2t}\right)}\\
&\stackrel{\mu}{\approx}\left(1\wedge \frac{(x-1)(y-1)}{t}\right)\frac{1}{\sqrt{t}}\left(\frac{y}{x}\right)^{\mu+1/2}\exp{\left(-\frac{(x-y)^2}{2t}\right)}\/.
}
This ends the proof.
\end{proof}

The proof of (\ref{eq:mainthm}) in the case $xy\geq t$ can be deduced from above-given propositions in the following way. Let $\mu\geq 1/2$ and without any loss of generality we assume that $x\leq y$. The upper bounds for every $x,y>1$ are given in Proposition \ref{prop:upperbounds:xytlarge}. From Proposition \ref{prop:xy:away} we know that the lower bounds are valid for $x,y>C_3^{(\mu)}$. If $x\leq C_3^{(\mu)}$ and $y\geq 5C_3^{(\mu)}(\mu+1)$ then the lower bounds are given in Proposition \ref{prop:xyt:large:lower:3}. Finally, taking $C=5C_3^{(\mu)}(\mu+1)$ in Corollary \ref{Cor:xy:bounded} we get the lower bounds in the remaining range of the parameters $x$ and $y$. The proof for $\nu\leq 1/2$ is obtained in the same way.

%%%%%%%%%%%%%%%%%%%%%%%%%%%%%%%%%%%%%%%%%%%

%						xy/t small

%%%%%%%%%%%%%%%%%%%%%%%%%%%%%%%%%%%%%%%%%%%

\section{Estimates for $xy/t$ small}
\label{section:xyt:small}
In this section we provide estimates of $p_1^{(\mu)}(t,x,y)$ whenever $xy<t$. Note also that (\ref{eq:mainthm}) can be written in the following shorter way
\formula{
   p_1^{(\mu)}(t,x,y) \stackrel{\mu}{\approx}\frac{x-1}{x}\frac{y-1}{y}\left(\frac{y^2}{t}\right)^{\mu+1/2}\frac{1}{\sqrt{t}}\exp\left(-\frac{x^2+y^2}{2t}\right)\/,
}
whenever $xy<t$. The main difficulty is to obtain the estimates when one of the space parameters is close to $1$ and the other is large, i.e. tends to infinity. In this case we have to take care of cancellations of two quantities  appearing in (\ref{eq:hunt:formula}) but also not to lose a control on the exponential behaviour. We begin with the upper bounds.
\begin{proposition}
\label{prop:xyt:small:upper}
   For every $\mu>0$, there exists constant $C_7^{(\mu)}>0$ such that 
   \formula{
       p_1^{(\mu)}(t,x,y)\leq C_7^{(\mu)} \frac{x-1}{x}\frac{y-1}{y}\left(\frac{y^2}{t}\right)^{\mu+1/2}\frac{1}{\sqrt{t}}\exp\left(-\frac{x^2+y^2}{2t}\right)\/,
   }
   whenever $xy\leq t$.
\end{proposition}
\begin{proof}
   If $x,y>2$ the result follows immediately from the general estimate $p_1^{(\mu)}(t,x,y)\leq p^{(\mu)}(t,x,y)$ and (\ref{eq:I:asym:zero}) which gives
   \formula[pt:estimate:xyt:small]{
      p^{(\mu)}(t,x,y)\approx \left(\frac{y^2}{t}\right)^{\mu+1/2}\frac{1}{\sqrt{t}}\exp\left(-\frac{x^2+y^2}{2t}\right)\/,\quad \frac{xy}{t}\leq 1\/.
   }
   Note that for every $x,y>0$ and $t>0$ there exists $c_1>0$ such that 
   \formula[pt:estimate:upper]{
      p^{(\mu)}(t,x,y)\leq c_1 \frac{y^{2\mu+1}}{t^{\mu+1}}\/.
   }
   If $xy<t$, then it immediately follows from (\ref{pt:estimate:xyt:small}) by estimating the exponential term by $1$. For $xy\geq t$ we use the asymptotic  behaviour (\ref{eq:I:asym:infty}) to show that
   \formula[pt:eestimate:xyt:large]{
      p^{(\mu)}(t,x,y)\approx\frac{1}{\sqrt{t}}\left(\frac{y}{x}\right)^{\mu+1/2}\exp\left(-\frac{|x-y|^2}{2t}\right)\leq \frac{y^{2\mu+1}}{t^{\mu+1}}\left(\frac{t}{xy}\right)^{\mu+1/2}\leq\frac{y^{2\mu+1}}{t^{\mu+1}}
   }
%   
%   Since the condition $xy<t$ implies $t>1$, the above-given formula can be rewritten for $1<x,y<2$ as 
%     \formula[pt:estimate:xyt:small:xybounded]{
%      p^{(\mu)}(t,x,y)\approx \frac{y^{2\mu+1}}{t^{\mu+1}}\/,\quad \frac{xy}{t}\leq 1\/.
%   }
%   for $xy\geq t$. However, taking into account (\ref{pt:estimate:xyt:small:xybounded}), we get the the above-given upper-bounds works for all $x,y>1$. 
   In particular, for all $z,w>1$ and $1<y<2$ there exists $c_2>0$ such that
   \formula[pt:estimate:rel]{
      p^{(\mu)}(t/3,z,w)\leq c_2\left(\frac{w}{y}\right)^{2\mu+1}\frac{1}{t^{\mu+1}}\/.
   }
  The Chapman-Kolmogorov equation and estimating the middle term using (\ref{pt:estimate:rel}) give
   \formula{
      p_1^{(\mu)}(t,x,y) &= \int_1^\infty\int_1^\infty p_1^{(\mu)}(t/3,x,z)p_1^{(\mu)}(t/3,z,w)p_1^{(\mu)}(t/3,w,y)dzdw\\
      &\leq \frac{c_3}{t^{\mu+1}} \int_1^\infty p_1^{(\mu)}(t/3,x,z)dz \int_1^\infty \left(\frac{w}{y}\right)^{2\mu+1} p_1^{(\mu)}(t/3,w,y)dw\\
     &= \frac{c_3}{t^{\mu+1}} P^{(\mu)}_x(T^{(\mu)}_1>t/3)P^{(\mu)}_y(T^{(\mu)}_1>t/3)\/.
   }
   Here the last equality follows from the symmetry property (\ref{eq:pt1:symmetry}).
   Since, by (\ref{sp:estimate:mu}), whenever $xy<t$ and $1<x,y<2$ we have
   \formula{
      P^{(\mu)}_x(T^{(\mu)}_1>t/3) &= P^{(\mu)}_x(\infty>T^{(\mu)}_1>t/3)+P^{(\mu)}_x(T^{(\mu)}_1=\infty)\approx \frac{x-1}{t^{\mu}}+1-\frac{1}{x^{2\mu}}\approx x-1\/,
   }
  which ends the proof of the upper-bound in this case. 
  
  Now assume that $y\geq 2$, $1<x\leq 2$ and $xy\leq t$. The other case $x\geq 2$, $1<y\leq 2$ follows from the symmetry condition mentioned above. Using the fact that $\int_0^\infty q_x^{(\mu)}(u)du=x^{-2\mu}$ and (\ref{eq:hunt:formula}), we can write
  \formula{
     p_1^{(\mu)}(t,x,y) &\leq p^{(\mu)}(t,x,y)-\int_0^{1/2}q_x^{(\mu)}(u)p^{(\mu)}(t-u,1,y)du\\
     &= J_1(t,x,y)+J_2(t,x,y)+J_3(t,x,y)\/,
  }
  where
  \formula{
     J_1(t,x,y) &= p^{(\mu)}(t,x,y)-\frac{1}{x^{2\mu}}p^{(\mu)}(t,x,y)+\pr^{(\mu)}_x(\infty>T_1^{(\mu)}>1/2)p^{(\mu)}(t,x,y)\/,\\
     J_2(t,x,y) &= \pr^{(\mu)}_x(T_1^{(\mu)}\leq 1/2)(p^{(\mu)}(t,x,y)-p^{(\mu)}(t,1,y))\/,\\
     J_3(t,x,y) &= \int_0^{1/2} q_x^{(\mu)}(u)(p^{(\mu)}(t,1,y)-p^{(\mu)}(t-u,1,y))\/,du\/.
   }
   It is obvious that for $1<x<2$ we have
   \formula{
     J_1(t,x,y)\leq c_4 (x-1)p^{(\mu)}(t,x,y)\/.
     }
   To deal with $J_2(t,x,y)$ note that the differentiation formula (\ref{eq:I:diff}), the asymptotic behavior (\ref{eq:I:asym:zero}) and positivity of $I_\mu(z)$ give 
   \formula{
      \dfrac{d}{dx}\left[e^{-x^2/2t}\left(\frac{t}{xy}\right)^{\mu}I_\mu\left(\frac{xy}{t}\right)\right] &= -\frac{x}{t}e^{-x^2/2t}\left(\frac{t}{xy}\right)^{\mu}I_\mu\left(\frac{xy}{t}\right)+e^{-x^2/2t}\frac{y}{t}\left(\frac{t}{xy}\right)^{\mu}I_{\mu+1}\left(\frac{xy}{t}\right)\\
      &\leq c_5 e^{-x^2/2t}\left(\frac{xy}{t}\right)^2\leq c_5\/,
   }
   whenever $xy<t$. Consequently, by mean value theorem, we obtain
   \begin{eqnarray*}
      J_2(t,x,y)\leq (p^{(\mu)}(t,x,y)-p^{(\mu)}(t,1,y))\leq c_5 (x-1)\left(\frac{y^2}{t}\right)^{\mu+1/2}\frac{1}{\sqrt{t}}e^{-y^2/2t}\/.
   \end{eqnarray*}
   Finally, the bounds of $J_3(t,x,y)$ follow from the estimates for the derivative of $p^{(\mu)}(t,1,y)$ in $t$. Using once again (\ref{eq:I:diff}) and skipping the negative components we have
   \begin{eqnarray*}
      h(t,y) &\stackrel{def}{=}& \dfrac{d}{dt}\left(\frac{1}{t^{\mu+1}}e^{-\frac{1+y^2}{2t}}\left(\frac{t}{y}\right)^{\mu}I_\mu\left(\frac{y}{t}\right)\right)\\
       &=& e^{-(1+y^2)/(2t)}\frac{I_\mu(y/t)}{ty^\mu}\left(-\frac{\mu+1}{t}+\frac{1+y^2}{2t^2}-\frac{y}{t}\frac{I_{\mu+1}(y/t)}{I_{\mu}(y/t)}\right)\\
      &\leq& e^{-(1+y^2)/(2t)}\frac{I_\mu(y/t)}{ty^\mu}\frac{1+y^2}{2t^2}\leq c_6 e^{-(1+y^2)/(2t)} \frac{1}{t^{\mu+1}}\/,
   \end{eqnarray*}
   whenever $y<t$. Thus, there exists $c=c_{\mu,u,y}\in(t-u,t)$ such that
   \begin{eqnarray*}
      J_3(t,x,y) &=& \int_0^{1/2}q_x^{(\mu)}(u)u y^{2\mu+1}h(c_{\mu,u,y},y)du \leq c_6 y^{2\mu+1}\int_0^{1/2}q_x^{(\mu)}(u)u e^{-(1+y^2)/(2c)} \frac{1}{c^{\mu+1}}du\\
      &\leq& c_6 e^{-(1+y^2)/(2t)}\frac{y^{2\mu+1}}{(t/2)^{\mu+1}}\int_0^{1/2}u q_{x}^{(\mu)}(u)du\/.
   \end{eqnarray*}
   Taking into account the upper bounds given in (\ref{hittingtime:estimates}) we get
   \formula{
      \int_0^{1/2}u q_{x}^{(\mu)}(u)du\leq c_7 \frac{x-1}{x^{\mu+1/2}}\int_0^{1/2}e^{-(x-1)^2/(2u)}\frac{du}{u^{1/2}}\leq c_8 (x-1)\/.
   }
    This ends the proof.
\end{proof}

The proof of the lower bounds is split into two parts. Next proposition corresponds to the case when $y>x>1$ and ${(y-1)^2}/{t}$ is large. Moreover, we enlarge the region and assume that $xy<mt$ for a given $m\geq 1$.  It is forced by the lower bounds given in Proposition \ref{prop:xyt:large:lower:3}, where it is required to have $xy/t$ sufficiently large but also by the proof of Proposition \ref{prop:xyt:small:lower:2}.
\begin{proposition}
\label{prop:xyt:small:lower:1}
For every $\mu>0$ and $m\geq 1$, there exists constant $C^{(\mu)}_8(m)>0$ such that 
\formula{
  \frac{p_1^{(\mu)}(t,x,y)}{p^{(\mu)}(t,x,y)}\geq C^{(\mu)}_8(m)\frac{x-1}{x}\/,\quad y>x>1
} 
whenever $xy<m t$ and $\frac{(y-1)^2}{t}\geq 2(\mu+1)$.
\end{proposition}
 \begin{proof}
 Since
\formula{
\frac{p_1^{(\mu)}(t,x,y)}{p^{(\mu)}(t,x,y)} &= 1-\frac{p^{(\mu)}(t,1,y)}{p^{(\mu)}(t,x,y)}\frac{r_1^{(\mu)}(t,x,y)}{p^{(\mu)}(t,1,y)}\/,
}
using (\ref{MBF:ineq:upper}) for every $\mu>0$ and $(y-1)^2/t\geq 2(\mu+1)$, we have
\formula{
\frac{r_1^{(\mu)}(t,x,y)}{p^{(\mu)}(t,1,y)} &= \int_0^t q_x^{(\mu)}(s)\frac{p^{(\mu)}(t-s,1,y)}{p^{(\mu)}(t,1,y)}\,ds\\
& = \int_0^t q_x^{(\mu)}(s)\frac{t}{t-s} \exp\left(-\frac{1+y^2}{2t}\frac{s}{t-s}\right)\frac{I_\mu(y/(t-s))}{I_\mu(y/t)}\,ds\\
& \leq \int_0^t q_x^{(\mu)}(s)\left(\frac{t}{t-s}\right)^{\mu+1}\exp\left(-\frac{(y-1)^2}{2t}\frac{s}{t-s}\right)\,ds\/.
}
For every $s<t$ we can write
\formula[eq:fw:ratio]{
  \left(\frac{t}{t-s}\right)^{\mu+1}\exp\left(-\frac{(y-1)^2}{2t}\frac{s}{t-s}\right) = \frac{f_y(t-s)}{f_y(t)}\/,
}
where $f_y(w) = w^{-\mu-1}e^{-(y-1)^2/2w}$. Then by simple calculation we get $f'_y(w) = w^{-\mu-2}e^{-(y-1)^2/2w}\left((\frac{(y-1)^2}{2w}-(\mu+1)\right)$ and consequently $f_y(w)$ is increasing on $\left(0,\frac{(y-1)^2}{2(\mu+1)}\right)$. It implies that right-hand side of (\ref{eq:fw:ratio}) is smaller than $1$ whenever $\frac{(y-1)^2}{t}\geq 2(\mu+1)$ and
\formula{
\frac{p_1^{(\mu)}(t,x,y)}{p^{(\mu)}(t,x,y)}&\geq 1-x^\mu\exp\left(\frac{x^2-1}{t}\right)\frac{I_\mu({y}/{t})}{I_\mu({xy}/{t})} \int_0^t q_x^{(\mu)}(s)\,ds\/.
}
Since the function $z^{-\mu}I_{\mu}(z)$ is increasing on $(0,\infty)$ 
\formula{
  \frac{I_\mu\left({y}/{t}\right)}{I_\mu\left({xy}/{t}\right)}\leq \frac{1}{x^\mu}\/,\quad x,y>1\quad t\geq 0\/.
}
This, together with $\pr^{(\mu)}_x(T_1^{(\mu)}<\infty)=x^{-2\mu}$, gives
\formula[LB:basic]{
\frac{p_1^{(\mu)}(t,x,y)}{p^{(\mu)}(t,x,y)}&\geq 1-\frac{1}{x^{2\mu}}\exp\left(\frac{x^2-1}{t}\right)\/.
}

Now we assume that $1<x<(2e^m)^{1/(2\mu)}$ and $t>\frac{2(2e^m)^{1/\mu}}{\mu}$. Then 
\formula{ 
\frac{p_1^{(\mu)}(t,x,y)}{p^{(\mu)}(t,x,y)}\geq 1-\frac{1}{x^{2\mu}}\exp\left(\frac{\mu(x^2-1)}{2(2e^m)^{1/\mu}}\right)
}
The mean value theorem ensures the existence of a constant  $d\in (1,x)$ such that
\formula{
1-\frac{1}{x^{2\mu}}\exp\left(\frac{\mu(x^2-1)}{2(2e^m)^{1/\mu}}\right) &= \frac{2\mu(x-1)}{d^{2\mu+1}}\exp\left(\frac{\mu(d^2-1)}{2(2e^m)^{1/\mu}}\right)\left(1-\frac{d^2}{2(2e^m)^{1/\mu}}\right)\\
&\geq c_1(m)(x-1)
}
where the last inequality comes from the fact that $1<d<x<(2e^m)^{1/(2\mu)}$.

The next step is to take $x\geq (2e^m)^{(1/(2\mu))}$ and $t>\frac{2(2e^m)^{1/\mu}}{\mu}$. Since $x^2<xy<mt$ using (\ref{LB:basic}) we get
\formula{
  \frac{p_1^{(\mu)}(t,x,y)}{p^{(\mu)}(t,x,y)}\geq 1-\frac{1}{x^{2\mu}}e^m\geq 1-\frac{1}{2} \approx \frac{x-1}{x}\/.
} 
Finally, we consider the case when $x>1$, $ xy/m<t\leq\frac{2(2e^m)^{1/\mu}}{\mu}=:t_0$ and $\frac{(y-1)^2}{t}\geq 2(\mu+1)$. Using absolute continuity property (\ref{ac:formula}) and (\ref{eq:pt1:12:formula}), we can write
   \formula{
      p_1^{(\mu)}(t,x,y)&\geq (e^{-t_0(\mu^2/2-1/8)}\wedge 1)\left(\frac{y}{x}\right)^{\mu-1/2}p_1^{(1/2)}(t,x,y)\\
      &\stackrel{\mu,m}{\approx}  \left(1\wedge\frac{(x-1)(y-1)}{t}\right)\left(\frac{y}{x}\right)^{\mu+1/2}\frac{1}{\sqrt{t}}\exp\left(-\frac{x^2+y^2}{2t}\right)\\
       &\geq  \left(1\wedge\frac{(x-1)\sqrt{2(\mu+1)/m}}{t_0}\right)\left(\frac{y^2}{t}\right)^{\mu+1/2}\frac{1}{\sqrt{t}}\exp\left(-\frac{x^2+y^2}{2t}\right)\\
       &\stackrel{\mu,m}{\approx}\frac{x-1}{x}p^{(\mu)}(t,x,y)\/.
   }
   This ends the proof.
\end{proof}
We end this section with the proof of the lower bounds, whenever ${((y\vee x)-1)^2}/{t}$ is small. Note that in the proof of the next proposition we use the lower bounds of $p_1^{(\mu)}(t,x,y)$ for $xy\geq t$ obtained previously in Section \ref{section:xyt:large} as well as the result of Proposition \ref{prop:xyt:small:lower:1}.  As previously, due to the symmetry, it is enough to assume that $y>x>1$.
\begin{proposition}
\label{prop:xyt:small:lower:2}
For every $\mu>0$ there exists constant $C_9^{(\mu)}>0$ such that
\formula{
\frac{p^{(\mu)}_1(t,x,y)}{p^{(\mu)}(t,x,y)}\geq C_9^{(\mu)} \frac{x-1}{x}\frac{y-1}{y}\/,\quad y>x>1\/,
}
whenever $xy<t$ and $\frac{(y-1)^2}{t}\leq 2(\mu+1)$.
\end{proposition}
\begin{proof}
   Let $xy<t$ and $y>x>1$. At the beginning we additionally assume that $t\geq 4$. Note that there exists $c_1>0$ such that for every $s>1/2$ we have $e^{-s}\geq c_1 s^{\mu+1/2}e^{-2s}$. This, together with the lower bounds of $p_1^{(\mu)}(t,z,w)$ for $z,w\geq \sqrt{t}$ (then $zw\geq t$) obtained in Section \ref{section:xyt:large}, enable us to write
   \formula{
      p_1^{(\mu)}(t,z,w)&\geq c_2\left(1\wedge\frac{(z-1)(w-1)}{t}\right)\left(\frac{w}{z}\right)^{\mu+1/2}\frac{1}{\sqrt{t}}\exp\left(-\frac{|z-w|^2}{2t}\right)\\
      &\geq \frac{c_2}{4}\left(\frac{w}{z}\right)^{\mu+1/2}\frac{1}{\sqrt{t}}\exp\left(-\frac{z^2}{2t}\right)\exp\left(-\frac{w^2}{2t}\right)\\
      &\geq \frac{c_2 c_1^2}{4} \left(\frac{wz}{t}\right)^{\mu+1/2}\left(\frac{w^2}{t}\right)^{\mu+1/2}\frac{1}{\sqrt{t}}\exp\left(-\frac{z^2}{t}\right)\exp\left(-\frac{w^2}{t}\right)\\
      &\geq c_3\left(\frac{w^2}{t}\right)^{\mu+1/2}\frac{1}{\sqrt{t}}\exp\left(-\frac{z^2}{t}\right)\exp\left(-\frac{w^2}{t}\right)\/.
   }
   Consequently, using the Chapmann-Kolmogorov equation and (\ref{eq:pt1:symmetry}), we get
   \formula{
      p_1^{(\mu)}\lefteqn{(3t,x,y) = \int_1^\infty\int_1^\infty p_1^{(\mu)}(t,x,z)p_1^{(\mu)}(t,z,w)p_1^{(\mu)}(t,w,y)dzdw}\\
      &\geq \int_{\sqrt{t}}^\infty\int_{\sqrt{t}}^\infty p_1^{(\mu)}(t,x,z)p_1^{(\mu)}(t,z,w)p_1^{(\mu)}(t,w,y)dzdw\\
      &\geq c_3\left(\frac{y^2}{t}\right)^{\mu+1/2}\frac{1}{\sqrt{t}}\int_{\sqrt{t}}^\infty p_1^{(\mu)}(t,x,z)e^{-{z^2}/{t}}dz \int_{\sqrt{t}}^\infty \left(\frac{w}{y}\right)^{2\mu+1}p_1^{(\mu)}(t,w,y)e^{-{w^2}/{t}}dw\\
      &= c_3\left(\frac{y^2}{t}\right)^{\mu+1/2}\frac{1}{\sqrt{t}}F^{(\mu)}_t(x)F^{(\mu)}_t(y)\/,
   }
   where 
   \formula{
     F^{(\mu)}_t(x) &:= \int_{\sqrt{t}}^\infty p_1^{(\mu)}(t,x,z)e^{-{z^2}/{t}}dz\/.
   }
   Since for $t\geq 4$ and $\frac{(y-1)^2}{t}\leq 2(\mu+1)$ we have
   \formula{
      \frac{x^2}{t}\leq \frac{y^2}{t}\leq \left(2\wedge 4\frac{(y-1)^2}{t}\right)\leq c_4
     }
    and consequently
    \formula{
       p^{(\mu)}(3t,x,y)\approx \left(\frac{y^2}{t}\right)^{\mu+1/2}\frac{1}{\sqrt{t}}\/,\quad xy<t\/,
    }
    it is enough to show that $F^{(\mu)}_t(x) \geq c_5 \frac{x-1}{x}$ for every $x>1$. However, for $z\geq b\sqrt{t}$, with $b=2\sqrt{2(\mu+1)}$, and $t\geq 4$ we have $\frac{(z-1)^2}{t}\geq \frac{1}{4}\frac{z^2}{t}\geq 2(\mu+1)$. We can use the lower bounds given in Proposition \ref{prop:xyt:small:lower:1} with $m=2b$ and obtain
    \formula{
      F^{(\mu)}_t(x) &\geq \int_{b\sqrt{t}}^{2bt/x} p_1^{(\mu)}(t,x,z)e^{-{z^2}/{t}}dz\\
      &\geq c_6 \frac{x-1}{x\sqrt{t}}\int_{b\sqrt{t}}^{2bt/x} \left(\frac{z^2}{t}\right)^{\mu+1/2}e^{-{z^2}/{2t}}e^{-{z^2}/{t}}dz
      \geq c_7\frac{x-1}{x\sqrt{t}}\int_{b\sqrt{t}}^{2bt/x} e^{-{2z^2}/{t}}dz\\
      &=  c_7\frac{x-1}{x}\int_{b}^{2b\sqrt{t}/x} e^{-{2u^2}}du\geq c_7\frac{x-1}{x}\int_{b}^{2b} e^{-{2u^2}}du\/.
    }
    Finally, for $t\leq 4$, the same computations as in the end of the proof of the previous Proposition (but with $t_0=4$) gives 
   \formula{
      p_1^{(\mu)}(t,x,y)&\geq c_7(e^{-4(\mu^2/2-1/8)}\wedge 1)\left(1\wedge\frac{(x-1)(y-1)}{t}\right)\left(\frac{y^2}{t}\right)^{\mu+1/2}\frac{1}{\sqrt{t}}\exp\left(-\frac{x^2+y^2}{2t}\right)\\
       &\stackrel{\mu}{\approx} \frac{x-1}{x}\frac{y-1}{y}p^{(\mu)}(t,x,y)\/,
   }
   where the last approximation follows from the fact that $(x-1)(y-1)< xy \leq t\leq 4$ which gives
   \formula{
    1\wedge \frac{(x-1)(y-1)}{t} = \frac{(x-1)(y-1)}{t}=\frac{(x-1)(y-1)}{xy}\frac{xy}{t}\approx\frac{(x-1)(y-1)}{xy}\/.
   }
      \end{proof}

\subsection*{Acknowledgments}
The authors are very grateful to Tomasz Byczkowski for critical remarks and comments
which enabled them to improve the presentation of the paper.

 \bibliography{bibliography}

\begin{thebibliography}{10}

\bibitem{BGS:2007}
T.~Byczkowski, P.~Graczyk, and A.~Stos.
\newblock Poisson kernels of half-spaces in real hyperbolic spaces.
\newblock {\em Rev. Mat. Iberoamericana}, 23(1):85--126, 2007.

\bibitem{BMR3:2013}
T.~Byczkowski, J.~Ma{\l}ecki, and M.~Ryznar.
\newblock Hitting times of {Bessel} processes.
\newblock {\em Potential Anal.}, 38:753--786, 2013.

\bibitem{BR:2006}
T.~Byczkowski and M.~Ryznar.
\newblock Hitting distibution of geometric {Brownian} motion.
\newblock {\em Studia Math.}, 173(1):19--38, 2006.

\bibitem{Davies:1987}
E.~B. Davies.
\newblock The equivalence of certain heat kernel and {Green} function bounds.
\newblock {\em J. Funct. Anal.}, 71:88--103, 1987.

\bibitem{Davies:1990}
E.~B. Davies.
\newblock {\em Heat kernels and spectral theory (Cambridge Tracts in
  Mathematics)}, volume~92.
\newblock Cambridge University Press, Cambridge, 1990.

\bibitem{Davies:1991}
E.~B. Davies.
\newblock Intrinsic ultracontractivity and the dirichlet laplacian.
\newblock {\em J. Funct. Anal.}, 100:162--180, 1991.

\bibitem{DaviesSimon:1984}
E.~B. Davies and B.~Simon.
\newblock Ultracontractivity and heat kernels for {Schrödinger} operators and
  {Dirichlet} {Laplacians}.
\newblock {\em J. Funct. Anal.}, 59:335--395, 1984.

\bibitem{Erdelyi:1953:volII}
Erdelyi et~al.
\newblock {\em Higher Transcendental Functions}, volume~II.
\newblock McGraw-Hill, New York, 1953.

\bibitem{GradsteinRyzhik:2007}
I.~S. Gradstein and I.~M. Ryzhik.
\newblock {\em Table of integrals, series and products. 7th edition}.
\newblock Academic Press, London, 2007.

\bibitem{HM:2013b}
Y.~Hamana and H.~Matsumoto.
\newblock Hitting times of {Bessel} processes, volume of {Wiener} sausages and
  zeros of {Macdonald} functions.
\newblock {\em arXiv:1302.4526}.

\bibitem{HM:2012}
Y.~Hamana and H.~Matsumoto.
\newblock The probability densities of the first hitting times of {Bessel}
  processes.
\newblock {\em J. Math-for-Ind.}, 4B:91--95, 2012.

\bibitem{HM:2013a}
Y.~Hamana and H.~Matsumoto.
\newblock The probability distributions of the first hitting times of {Bessel}
  processes.
\newblock {\em Trans. Amer. Math. Soc}, 365:5237--5257, 2013.

\bibitem{Laforgia:1991}
A.~Laforgia.
\newblock Bounds for modified {Bessel} functions.
\newblock {\em J. Comput. Appl. Math.}, 34:263--267, 1991.

\bibitem{MatsumotoYor:2005a}
H.~Matsumoto and M.~Yor.
\newblock Exponential functionals of {Brownian} motion, {I}: {Probability} laws
  at fixed time.
\newblock {\em Probability Surveys}, 2:312--347, 2005.

\bibitem{MatsumotoYor:2005b}
H.~Matsumoto and M.~Yor.
\newblock Exponential functionals of {Brownian} motion, {II}: {Some} related
  diffusion processes.
\newblock {\em Probability Surveys}, 2:348--384, 2005.

\bibitem{NowakRoncal:2013a}
A.~Nowak and L.~Roncal.
\newblock On sharp heat and subordinated kernel estimates in the
  {Fourier-Bessel} setting.
\newblock {\em arXiv:1111.5700}.

\bibitem{NowakRoncal:2013b}
A.~Nowak and L.~Roncal.
\newblock Sharp heat kernel estimates in the {Fourier-Bessel} setting for a
  continuous range of the type parameter.
\newblock {\em arXiv:1208.5199}.

\bibitem{SC:2010}
L.~Saloff-Coste.
\newblock The heat kernel and its estimates.
\newblock {\em Adv. Stud. Pure Math.}, 57:405--436, 2010.

\bibitem{Zhang:2002}
Q.~S. Zhang.
\newblock The boundary behavior of heat kernels of {Dirichlet} {Laplacians}.
\newblock {\em J. Differential Equations}, 182:416--430, 2002.

\end{thebibliography}
\bibliographystyle{plain}
\end{document}